\documentclass[12pt]{article}
\usepackage{amsmath,amsopn,amssymb,amsthm,amsfonts}
\usepackage[T2A]{fontenc}
\usepackage[cp1251]{inputenc}
\usepackage{longtable}

\newtheorem{theorem}{Theorem}[section]

\newcommand\g{{\mathfrak g}}
\newcommand\h{{\mathfrak h}}

\newcommand\dd{{\mathfrak d}}

\newcommand\rr{{\mathfrak r}}

\newcommand{\R}{\mathbb{R}}

\begin{document}

{\bf \large
\centerline{N.~K.~Smolentsev, I.~Y.~Shagabudinova}

\vspace{3mm}
\centerline{On the classification of left-invariant para-K\"{a}hler structures}
\centerline{on four-dimensional Lie groups}
\vspace{3mm}
}
\begin{abstract}
A first classification of para-K\"{a}hler structures on four-dimensional Lie algebras was obtained by D. Calvaruzo in 2015. In this paper, we propose another classification based on the classification of symplectic Lie algebras. For each four-dimensional symplectic Lie algebra, compatible para-complex structures and the corresponding pseudo-Riemannian metrics are found in explicit form. This leads to the classification of Sasaki para-contact structures on five-dimensional contact Lie algebras with a nontrivial center.
\end{abstract}

\section{Introduction} \label{Intro}
Complex manifolds are among the most active research fields in differential geometry.
In recent years, their paracomplex analogue attracted a growing number of researchers \cite{Aleks}, \cite{Calva-Fino}, \cite{Cru-Fort-Gadea}, \cite{Smolen-19}.
If we consider an even-dimensional Lie group, it is a natural problem to determine its invariant complex and paracomplex structures.
The structures on Lie groups of small dimensions are of particular interest in view of the possibilities of their classification.
In particular, four-dimensional complex Lie algebras, symplectic and pseudo-K\"{a}hler Lie algebras were classified by Owando \cite{Ovando2000} - \cite{Ovando2006-1}.
A classification of para-K\"{a}hler structures on four-dimensional Lie groups was obtained by D. Calvaruso \cite{Calvar-15} based on the classification of product structures \cite{And-Barbe-Dotti-Ovando}.
For each given product structure $J$, Calvaruso found all possible symplectic structures $\omega$ compatible with $J$.

In this paper, we return to the classification of four-dimensional para-K\"{a}hler Lie algebras for the following reasons.
First, it seems more natural for us to proceed from a given symplectic structure $\omega$ on a Lie algebra $\g$ and to search all compatible para-complex structures $J$ for it.
Secondly, with this approach, we obtain a series of decompositions of the Lie algebra $\g$ into a direct sum of subalgebras, $\g = \g_+\oplus \g_-$.
And finally, the explicit form of para-complex structures $J$ on symplectic Lie algebras allows us possible to explicitly obtain para-contact metric structures on five-dimensional Lie algebras, which are central extensions of the symplectic Lie algebras.
In this case, the operators $J$ of para-complex structures determine the affinors $\varphi$ of para-contact structures.
Therefore, the classification of four-dimensional para-K\"{a}hler Lie algebras obtained with this approach leads to the classification of Sakaki parakontact structures on five-dimensional contact Lie algebras with a nontrivial center.

In this paper, for each four-dimensional symplectic Lie algebra $(\g, \omega)$ we find all para-complex structures $J$ compatible with $\omega$.
As a result, we obtain para-K\"{a}hler Lie algebras $(\g, \omega, J, g)$ with a pseudo-Riemannian metric $g = \omega \cdot J$.
In addition, we find the geometric characteristics of each structure $(\g, \omega, J, g)$.

All calculations were made in the Maple system according to the usual formulas for the geometry of left-invariant structures.

\section{Preliminaries} \label{Preface}
A (1,1)-tensor $J$ on manifold $M$ is called an {\it almost paracomplex} structure if $J^2= Id$ and its eigenvalues $-1$ and $+1$ have the same rank.
The integrability of $J$ is expressed by condition
\begin{equation}\label{N_J}
N_J(X,Y) = [X,Y] + [JX,JY] -J[JX,Y] -J[X,JY] = 0,
\end{equation}
for all tangent vector fields $X, Y$.
When (\ref{N_J}) holds, $J$ is called a {\it para-complex} structure.
A pseudo-Riemannian metric $g$ on $M$ is said to be compatible with $J$ if
\begin{equation}\label{g_J}
g(JX,Y) + g(X,JY) = 0,
\end{equation}
for all tangent vector fields $X, Y$.
In such a case, the pair $(g,J)$ is called an {\it almost para-Hermitian} structure.
It is well known that the $g$ is necessarily of neutral signature $(n,n)$.

When (\ref{g_J}) holds, one can consider the fundamental two-form $\omega(X,Y):= g(X,JY)$ of the almost para-Hermitian structure.
If $\omega$ is symplectic, then $(M,g,J,\omega)$ is said to be an {\it almost para-K\"{a}hler} manifold.
In particular, a para-K\"{a}hler structure is formed by a pseudo-Riemannian metric $g$, a compatible para-complex structure $J$ and the associated symplectic form $\omega$.
We explicitly observe that because of the compatibility conditions
\begin{equation}\label{compatible}
\omega(X,Y) = g(X,JY),\ g(X,Y) = \omega(X,JY),\ \omega(JX,JY) + \omega(X,Y) = 0,
\end{equation}
any two among $g$, $J$ and $\omega$ uniquely determine the third one.
For this reason, we simply denote a para-K\"{a}hler structure by $(\omega,J)$.

Let now $G$ denote a $2n$-dimensional (simply connected) Lie group and $\g$ its Lie algebra. If $(g,J)$ is a left-invariant para-K\"{a}hler structure on $G$, then tensors $g$ and $J$ are uniquely determined at the algebraic level by corresponding tensors (denoted again by the same symbols) defined on the Lie algebra $\g$. More precisely, $J$
and $g$ respectively correspond to an endomorphism and an inner product on $\g$.
Then $(\g, \omega, J, g)$ is called the para-K\"{a}hler Lie algebra.
Note that any para-complex structure $J$ on the Lie algebra $\g$ yields a decomposition of $\g$ into a direct sum of subalgebras:
$$
\g = \g_+\oplus g_-,\ \text{where}\ J|_{\g_+} = Id, J|_{\g_-} = -Id.
$$

The four-dimensional symplectic Lie algebras were found in \cite{Ovando2006}.
The para-complex structures on four-dimensional solvable Lie algebras were obtained in \cite{And-Barbe-Dotti-Ovando}. In article \cite{Calvar-15}, all possible symplectic structures are found that are compatible with a para-complex structure.
Given all these results, we obtain the following list of four-dimensional symplectic Lie algebras admitting a para-K\"{a}hler structure (Table 1).
This table shows the nonzero Lie brackets of algebras in the basis $e_1, e_2, e_3, e_4$ and the symplectic structures in the dual basis $e^1, e^2, e^3, e^4$.

\begin{longtable}[H]{|l|l|l|}
\hline
$\g$ & Lie brackets & Symplectic structure\\

\hline
$\rr_{2}\rr_{2}$ & $[e_1,e_2] = e_2, [e_3,e_4] = e_4$ &
$\omega = e^1\wedge e^2 + \lambda e^1\wedge e^3 + e^3\wedge e^4 , \lambda \ge 0$
\\

\hline
$\rr\h_{3}$ & $[e_1,e_2] = e_3$ &
$\omega = e^1\wedge e^4 + e^2\wedge e^3$
\\

\hline
$\rr\rr_{3,0}$ & $[e_1,e_2] = e_2, [e_1,e_3] = 0$ &
$\omega = e^1\wedge e^2 + e^3\wedge e^4$
\\

\hline
$\rr\rr_{3,-1}$ & $[e_1,e_2] = e_2, [e_1,e_3] = -e_3$ &
$\omega = e^1\wedge e^4 + e^2\wedge e^3$
\\

\hline
$\rr_{2}'$ &
\begin{tabular}{l}
  $[e_1,e_3] = e_3, [e_1,e_4] = e_4, [e_2,e_3] = e_4$, \\
  $[e_2,e_4] = -e_3$ \\
\end{tabular} &
$\omega = e^1\wedge e^4 + e^2\wedge e^3$
\\

\hline
$\rr_{4,0}$ &
$[e_4,e_1] = e_1, [e_4,e_2] =0, [e_4,e_3] = e_2$ &
\begin{tabular}{l}
  $\omega_+ = e^1\wedge e^4 + e^2\wedge e^3$, \\
  $\omega_- = e^1\wedge e^4 -e^2\wedge e^3$ \\
\end{tabular}
 \\

\hline
$\rr_{4,-1}$ &
$[e_4,e_1] = e_1, [e_4,e_2] =-e_2, [e_4,e_3] = e_2-e_3$ &
$\omega = e^1\wedge e^3 + e^2\wedge e^4$
\\

\hline
$\rr_{4,-1,\beta}$ &
$[e_4,e_1] = e_1, [e_4,e_2] =-e_2, [e_4,e_3] = \beta e_3$ &
$\omega = e^1\wedge e^2 + e^3\wedge e^4$, $-1<\beta < 0$
\\

\hline
$\rr_{4,-1,-1}$ &
$[e_4,e_1] = e_1, [e_4,e_2] =-e_2, [e_4,e_3] = -e_3$ &
$\omega = e^1\wedge e^2 + e^3\wedge e^4$
\\

\hline
$\rr_{4,-\alpha,\alpha}$ &
$[e_4,e_1] = e_1, [e_4,e_2] =-\alpha e_2, [e_4,e_3] = \alpha e_3$ &
$\omega = e^1\wedge e^4 + e^2\wedge e^3$, $0<\alpha <1$
\\

\hline
$\dd_{4,1}$ &
\begin{tabular}{l}
  $[e_1,e_2] = e_3, [e_4,e_3] =e_3, [e_4,e_1] =e_1$, \\
  $[e_4,e_2] =0$ \\
\end{tabular} &
\begin{tabular}{l}
  $\omega_1 = e^1\wedge e^2 -e^3\wedge e^4$, \\
  $\omega_2 = e^1\wedge e^2 -e^3\wedge e^4 +e^2\wedge e^4$ \\
\end{tabular}
\\

\hline
$\dd_{4,2}$ &
\begin{tabular}{l}
  $[e_1,e_2] = e_3, [e_4,e_3] =e_3, [e_4,e_1] =2e_1$, \\
  $[e_4,e_2] =-e_2$ \\
\end{tabular} &
\begin{tabular}{l}
  $\omega_1 = e^1\wedge e^2 -e^3\wedge e^4$, \\
  $\omega_2 = e^1\wedge e^4 +e^2\wedge e^3$, \\
  $\omega_3 = e^1\wedge e^4 -e^2\wedge e^3$ \\
\end{tabular}
\\

\hline
$\dd_{4,\lambda}$ &
\begin{tabular}{l}
  $[e_1,e_2] = e_3, [e_4,e_3] =e_3, [e_4,e_1] =\lambda e_1$, \\
  $[e_4,e_2] =(1-\lambda)e_2$ \\
\end{tabular} &
$\omega = e^1\wedge e^2 -e^3\wedge e^4$, $\lambda\ge \frac 12, \lambda\neq 1, 2$
\\

\hline
$\h_{4}$ &
\begin{tabular}{l}
   $[e_1,e_2] = e_3, [e_4,e_3] =e_3, [e_4,e_1] =\frac 12 e_1$, \\
  $[e_4,e_2] =e_1+\frac 12 e_2$ \\
\end{tabular} &

\begin{tabular}{l}
  $\omega_+ = e^1\wedge e^2 -e^3\wedge e^4$, \\
  $\omega_- = -e^1\wedge e^2 +e^3\wedge e^4$ \\
\end{tabular}
\\

\hline
$\mathbb{R}^{4}$ &
&
$\omega = e^1\wedge e^2 +e^3\wedge e^4$
\\
\hline
\caption{Four-dimensional symplectic Lie algebras admitting a para-K\"{a}hler structure}
\end{longtable}

In the above list, $\rr_2 = aff(\mathbb{R})$ is the Lie algebra of the Lie group of affine motions of $\mathbb{R}$;
$\rr_{2}'$ is the real Lie algebra underlying to the complex Lie algebra $aff(\mathbb{C})$; $\rr\rr_{3,0}$, $\rr\rr_{3,-1}$ and $\rr\h_{3}$ are the trivial extensions of the Lie algebra $\mathfrak{e}(2)$ of the group of rigid motions of $\mathbb{R}^2$, the Lie algebra $\mathfrak{e}(1,1)$ of the group of rigid motions of the Minkowski two-space and of the Heisenberg Lie algebra $\h_{3}$, respectively.

Let $\nabla$ be the Levi-Civita connection corresponding to a left-invariant pseudo-Riemannian metric $g$.
Recall that curvature tensor $R(X,Y) = [\nabla_X, \nabla_Y] -\nabla_{[X,Y]}$ is calculated by the formula  $R_{ijk}^{s}=\Gamma_{ip}^{s}\Gamma_{jk}^{p}-\Gamma_{jp}^{s}\Gamma_{ik}^{p} -C_{ij}^{p}\Gamma_{pk}^{s}$, where $\Gamma_{ij}^m = \frac 12 g^{km}\left (C_{ij}^p g_{pk} +C_{ki}^p g_{pj} +C_{kj}^p g_{ip} \right )$ are Christoffel symbols and $C_{ij}^p$ are structure constants of $\g$.
The Ricci tensor is a convolution of the curvature tensor for the first and fourth (upper) indices $Ric_{jk}=R_{ijk}^i$.

\section{Classification of four-dimensional para-K\"{a}hler Lie algebras} \label{Lie_algebras}

\begin{theorem} \label{T1}
A four-dimensional para-K\"{a}hler Lie algebra is isomorphic to one of the Lie algebras, endowed with symplectic and compatible para-complex structure as listed in the Sections \ref{G_1} -- \ref{G_15}.
\end{theorem}

\begin{proof}
For each four-dimensional symplectic Lie algebra $(\g,\omega)$ from the Table 1, we find all compatible para-complex structures $J$, that is, such endomorphisms $J:\g\to \g$ that have the properties:
\begin{enumerate}
  \item $J^2=Id$.
  \item $\omega(JX,Y)+\omega(X,JY)=0$.
  \item $[X,Y] + [JX,JY] -J[JX,Y] -J[X,JY] = 0$.
\end{enumerate}
Then $g(X,Y) = \omega(X,JY)$ is a pseudo-Riemannian metric of signature (2,2) and $(\omega, J, g)$ forms a para-K\"{a}hler structure on the Lie algebra $\g$.
It is clear that the para-complex structure $J$ is determined up to a sign. We will indicate only one $J$ of two $\pm J$.
The para-complex structures $J$ will be represented by matrices $J=J^i_j e_i\otimes e^j$ in the basis $e_1, e_2, e_3, e_4$ of the Lie algebra $\g$.
Then the integrability condition 3 has the form $J^l_i J^m_j C_{lm}^k - J^l_i J^k_m C_{lj}^m - J^l_j J^k_m C_{il}^m + C_{ij}^k =0$.
If $J$ is a para-complex structure compatible with $\omega$, then the metric tensor $g$ can be found by the formula $g=\omega\cdot J$, i.e. $g_{ij}=\omega_{jk}J^k_j$.
For each para-K\"{a}hler structure $(\g, \omega, J, g)$, the curvature and Ricci tensors are calculated in the Maple system. We represent the Ricci operator $RIC = Ric\cdot g^{-1}$ from which the Ricci tensor and scalar curvature $S$ is easily found.
\end{proof}

\subsection{Lie algebra $\rr_2\rr_2$} \label{G_1}
Nonzero Lie bracket: $[e_1,e_2] = e_2, [e_3,e_4] = e_4$. Symplectic structure: $\omega = e^1\wedge e^2 + \lambda e^1\wedge e^3 + e^3\wedge e^4$, $\lambda \ge 0$.
Cases $\lambda > 0$ and $\lambda =0$ significantly different.

\subsubsection{The case $\lambda > 0$} There are three para-complex structures $J$, that, together with the 2-form $\omega$ give three para-K\"{a}hler structures depending on the parameters $a$ and $b$, all of zero curvature:
There are three para-complex structures $J$, which together with the 2-form $\omega $ and metric $g=\omega\cdot J$ give three para-K\"{a}hler structures, all of zero curvature for any values of parameters $a$ and $b$:
$$
J_{1,1}= \left[ \begin {array}{cccc} -1&0&0&0\\ \noalign{\medskip}{\it a}&
1&0&0\\ \noalign{\medskip}0&0&1&0\\ \noalign{\medskip}0&0&{\it b}&
-1\end {array} \right], \quad
J_{1,2}= \left[ \begin {array}{cccc} -1&0&2&0\\ \noalign{\medskip}-{\it a}
&1&{\it a}&0\\ \noalign{\medskip}0&0&1&0\\ \noalign{\medskip}{\it
a}&-2&{\it b}&-1\end {array} \right],\quad
J_{1,3}=  \left[ \begin {array}{cccc} -1&0&0&0\\ \noalign{\medskip}{\it a}&
1&{\it b}&2\\ \noalign{\medskip}-2&0&1&0\\ \noalign{\medskip}{\it
b}&0&-{\it b}&-1\end {array} \right].
$$

\subsubsection{The case $\lambda =0$}
There are 5 para-K\"{a}hler structures:

1.	The para-K\"{a}hler structure with Hermitian Ricci tensor (i.e. $Ric(JX,JY) = Ric(X,Y)$) and with the Ricci operator $RIC = Ric\cdot g^{-1}$ and scalar curvature $S={\rm trace}(RIC)$:
$$
J_{2,1}= \left[ \begin {array}{cccc} -{\it a}&{\it b}&0&0
\\ \noalign{\medskip}-{\frac {{{\it a}}^{2}-1}{{\it b}}}&{\it
a}&0&0\\ \noalign{\medskip}0&0&{\it c}&-{\frac {{{\it c}}^
{2}-1}{{\it d}}}\\ \noalign{\medskip}0&0&{\it d}&-{\it c}
\end {array} \right], \quad
RIC_{2,1}= \left[ \begin {array}{cccc} -{\it b}&0&0&0\\ \noalign{\medskip}0&
-{\it b}&0&0\\ \noalign{\medskip}0&0&{\frac {{{\it c}}^{2}-1}{
{\it d}}}&0\\ \noalign{\medskip}0&0&0&{\frac {{{\it c}}^{2}-1}
{{\it d}}}\end {array} \right].$$

2.	The Einstein para-Kahler structure:
$$
J_{2,2}= \left[ \begin {array}{cccc} {\it a}+1&{\it b}&-1+{\it c}&
{\it b}\\ \noalign{\medskip}-{\frac {{\it a}\, \left( {\it
a}+2 \right) }{{\it b}}}&-{\it a}-1&-{\frac { \left( -1+{
\it c} \right) {\it a}}{{\it b}}}&-{\it a}
\\ \noalign{\medskip}{\it a}&{\it b}&{\it c}&{\it b}
\\ \noalign{\medskip}-{\frac { \left( -1+{\it c} \right) {\it
a}}{{\it b}}}&1-{\it c}&-{\frac {{{\it c}}^{2}-1}{{
\it b}}}&-{\it c}\end {array} \right], \quad
RIC_{2,2}= -\frac {3b}{2} Id.
$$

3. The para-K\"{a}hler structure with Hermitian Ricci tensor:
$$
J_{2,3}= \left[ \begin {array}{cccc} 1&0&0&0\\ \noalign{\medskip}{\it a}&-
1&0&0\\ \noalign{\medskip}0&0&{\it b}&-{\frac {{{\it b}}^{2}-1
}{{\it c}}}\\ \noalign{\medskip}0&0&{\it c}&-{\it b}
\end {array} \right],\qquad
RIC_{2,3}= \left[ \begin {array}{cccc} 0&0&0&0\\ \noalign{\medskip}0&0&0&0
\\ \noalign{\medskip}0&0&{\frac {{{\it b}}^{2}-1}{{\it c}}}&0
\\ \noalign{\medskip}0&0&0&{\frac {{{\it b}}^{2}-1}{{\it c}}}
\end {array} \right].
$$

4.	The para-K\"{a}hler structure with Hermitian Ricci tensor. Einstein for $b = c$:
$$
J_{2,4}=  \left[ \begin {array}{cccc} -{\it a}&{\it b}&0&0
\\ \noalign{\medskip}-{\frac {{{\it a}}^{2}-1}{{\it b}}}&{\it
a}&0&0\\ \noalign{\medskip}0&0&-1&{\it c}\\ \noalign{\medskip}0
&0&0&1\end {array} \right], \quad
RIC_{2,4}= \left[ \begin {array}{cccc} -{\it b}&0&0&0\\ \noalign{\medskip}0&
-{\it b}&0&0\\ \noalign{\medskip}0&0&-{\it c}&0
\\ \noalign{\medskip}0&0&0&-{\it c}\end {array} \right].
$$

5.	The para-K\"{a}hler structure of zero curvature:
$$
J_{2,5}= \left[ \begin {array}{cccc} 1&0&0&0\\ \noalign{\medskip}{\it a}&-
1&{\it b}&-2\\ \noalign{\medskip}2&0&-1&0\\ \noalign{\medskip}{
\it b}&0&-{\it b}&1\end {array} \right],\quad
g_{2,5}= \left[ \begin {array}{cccc} {\it a}&-1&{\it b}&-2
\\ \noalign{\medskip}-1&0&0&0\\ \noalign{\medskip}{\it b}&0&-{\it
b}&1\\ \noalign{\medskip}-2&0&1&0\end {array} \right].
$$

\subsection{Lie algebra $\rr\h_3$} \label{G_2}
Nonzero Lie bracket: $[e_1,e_2] = e_3$. Symplectic structure: $\omega = e^1\wedge e^4 + e^2\wedge e^3$. There are two para-K\"{a}hler structures of zero curvature:
$$
J_1= \left[ \begin {array}{cccc} {\it a}&-{\it b}&0&0
\\ \noalign{\medskip}{\frac {{{\it a}}^{2}-1}{{\it b}}}&-{\it
a}&0&0\\ \noalign{\medskip}{\it d}&{\it c}&{\it a}&{
\it b}\\ \noalign{\medskip}-{\frac {{\it c}\,{{\it a}}^{2}
+2\,{\it d}\,{\it a}\,{\it b}-{\it c}}{{{\it b}}^{
2}}}&{\it d}&-{\frac {{{\it a}}^{2}-1}{{\it b}}}&-{\it
a}\end {array} \right], \quad
J_2= \left[ \begin {array}{cccc} 1&-{\it b}&0&0\\
\noalign{\medskip}0&-1&0&0\\ \noalign{\medskip}-\frac{bd}{2}&{\it c}&
1&{\it b}\\ \noalign{\medskip}{\it d}&-\frac{bd}{2}&0&-1\end {array} \right].
$$

\subsection{Lie algebra $\rr\rr_{3,0}$} \label{G_3}
Nonzero Lie bracket: $[e_1,e_2] = e_2, [e_1,e_3] = 0$. Symplectic structure: $\omega = e^1\wedge e^2 + e^3\wedge e^4$. Two para-complex structures J.

1.	The para-K\"{a}hler structure with Hermitian Ricci tensor
$$
J_1= \left[ \begin {array}{cccc} -{\it a}&{\it c}&0&0
\\ \noalign{\medskip}-{\frac {{{\it a}}^{2}-1}{{\it c}}}&{\it
a}&0&0\\ \noalign{\medskip}0&0&{\it b}&-{\frac {{{\it b}}^
{2}-1}{{\it d}}}\\ \noalign{\medskip}0&0&{\it d}&-{\it b}
\end {array} \right],\quad
RIC_1 = \left[ \begin {array}{cccc} -{\it c}&0&0&0\\ \noalign{\medskip}0&
-{\it c}&0&0\\ \noalign{\medskip}0&0&0&0\\ \noalign{\medskip}0&0&0
&0\end {array} \right].
$$

2.	The para-K\"{a}hler structure of zero curvature
$$
J_2= \left[ \begin {array}{cccc} -1&0&0&0\\ \noalign{\medskip}{\it a}&
1&0&0\\ \noalign{\medskip}0&0&{\it b}&-{\frac {{{\it b}}^{2}-1
}{{\it c}}}\\ \noalign{\medskip}0&0&{\it c}&-{\it b}
\end {array} \right], \quad
g_2= \left[ \begin {array}{cccc} {\it a}&1&0&0\\ \noalign{\medskip}1&0
&0&0\\ \noalign{\medskip}0&0&{\it c}&-{\it b}
\\ \noalign{\medskip}0&0&-{\it b}&{\frac {{{\it b}}^{2}-1}{{
\it c}}}\end {array} \right].
$$

\subsection{Lie algebra $\rr\rr_{3,-1}$} \label{G_4}
Nonzero Lie bracket: $[e_1,e_2] = e_2$, $[e_1,e_3] = -e_3$.
Symplectic structure: $\omega =e^1\wedge e^4 +e^2\wedge e^3$.
Three compatible para-complex structures $J$.

Two para-K\"{a}hler structures with zero Ricci tensor:
$$
J_1= \left[ \begin {array}{cccc} 1&0&0&0\\ \noalign{\medskip}0&-1&0&0
\\ \noalign{\medskip}0&{\it a}&1&0\\ \noalign{\medskip}{\it b}
&0&0&-1\end {array} \right], \quad
J_2= \left[ \begin {array}{cccc} -1&0&0&0\\ \noalign{\medskip}0&-1&{\it
a}&0\\ \noalign{\medskip}0&0&1&0\\ \noalign{\medskip}{\it b}&0
&0&1\end {array} \right].
$$

Para-K\"{a}hler structure of zero curvature:
$$
J_3= \left[ \begin {array}{cccc} -{\it a}&0&0&-{\frac {{{\it a}}^{
2}-1}{{\it b}}}\\ \noalign{\medskip}0&-1&0&0\\ \noalign{\medskip}0
&0&1&0\\ \noalign{\medskip}{\it b}&0&0&{\it a}\end {array}
 \right], \quad
g_3= \left[ \begin {array}{cccc} {\it b}&0&0&{\it a}
\\ \noalign{\medskip}0&0&1&0\\ \noalign{\medskip}0&1&0&0
\\ \noalign{\medskip}{\it a}&0&0&{\frac {{{\it a}}^{2}-1}{{
\it b}}}\end {array} \right].
$$

\subsection{Lie algebra $\rr_{2}'$} \label{G_5}
Nonzero Lie bracket: $[e_1,e_3] = e_3$, $[e_1,e_4] = e_4$, $[e_2,e_3] = e_4$, $[e_2,e_4] = -e_3$.
Symplectic structure: $\omega =e^1\wedge e^4 +e^2\wedge e^3$.
Three compatible para-complex structures $J$:

1.	The para-Kahler structure with Hermitian Ricci tensor
$$
J_1= \left[ \begin {array}{cccc} -{\it a}&{\it b}&{\it c}&-{\it d}\\
\noalign{\medskip}-{\it b}&-{\it a}&{\it d}&{\it c}\\ \noalign{\medskip}J_1^3&J_{32}&{\it a}&-{\it b}\\
\noalign{\medskip}J_{41} & J_2^4  &{\it b}&{\it a}\end {array}
 \right], \quad
RIC_1=\left[ \begin {array}{cccc} 2\,{\it d}&-2\,{\it c}&0&0\\ \noalign{\medskip}2\,{\it c}&2\,{\it d}&0&0 \\
\noalign{\medskip}0&0&2\,{\it d}&-2\,{\it c}\\
\noalign{\medskip}0&0&2\,{\it c}&2\,{\it d}\end {array}
 \right],
$$
where $J_1^3= J_2^4= -\frac {2\,{\it d}\,{\it a}\,{
\it b}+ c(a^{2}-b^{2}-1)}{{{\it d}}^{2}+{{\it c}}^{2}}$ and $J_2^3= -J_1^4 =\frac {2\,{\it c}\,{\it a}\,{\it b} -d(a^{2}-b^{2}-1)} {{{\it d}}^{2}+{{\it c }}^{2}}$.

2.	The para-K\"{a}hler Einstein structure
$$
J_2= \left[ \begin {array}{cccc} -\frac{a\, b+c^{2}+2}{2}&-{\it c}&0&{\it b}\\ \noalign{\medskip}0&-1&0&0\\
\noalign{\medskip}-{\frac {{\it c}\, \left( {\it a}\,{\it b}+{{\it c}}^{2}+4 \right) }{{2\, b}}}&{\it a}&1&{\it c}\\
\noalign{\medskip}-{\frac {{{\it a}}^{2}{{\it b}}^{2}+2\,{\it a}\,{\it b}\,{{\it c}}^{2}+{{\it c}}^{4}+4\,{\it a}\,{\it b}+4\,{{\it c}}^{2}}{{4\, b}}}&-{\frac {{\it c}\, \left( {\it a}\,{\it b}+{{\it c}}^{2}+4 \right) }{{2\, b}}}&0&\frac{a\, b+c^{2}+2}{2}\end {array} \right], \quad
RIC_2= -\frac {3b}{2}\, Id.
$$

3.	The para-K\"{a}hler structure of zero curvature
$$
J_3= \left[ \begin {array}{cccc} -1&0&0&0\\ \noalign{\medskip}0&-1&0&0
\\ \noalign{\medskip}{\it a}&{\it b}&1&0\\ \noalign{\medskip}-
{\it b}&{\it a}&0&1\end {array} \right], \quad
g_3= \left[ \begin {array}{cccc} -{\it b}&{\it a}&0&1
\\ \noalign{\medskip}{\it a}&{\it b}&1&0\\ \noalign{\medskip}0
&1&0&0\\ \noalign{\medskip}1&0&0&0\end {array} \right].
$$

\subsection{Lie algebra $\rr_{4,0}$} \label{G_6}
Nonzero Lie bracket: $[e_4,e_1] = e_1$, $[e_4,e_3] = e_2$.
Symplectic structures: $\omega_+ =e^1\wedge e^4 +e^2\wedge e^3$ and $\omega_- =e^1\wedge e^4 -e^2\wedge e^3$.
Two para-K\"{a}hler structures with zero Ricci tensor: .
$$
J_+= -J_-=\ \left[ \begin {array}{cccc} -1&0&0&{\it a}\\ \noalign{\medskip}0&
1&{\it b}&0\\ \noalign{\medskip}0&0&-1&0\\ \noalign{\medskip}0&0&0
&1\end {array} \right].
$$

\subsection{Lie algebra $\rr_{4,-1}$} \label{G_7}
Nonzero Lie bracket: $[e_4,e_1] = e_1$, $[e_4,e_2] =-e_2$, $[e_4,e_3] = e_2-e_3$.
Symplectic structure: $\omega =e^1\wedge e^3 +e^2\wedge e^4$.
One para-K\"{a}hler structure with zero Ricci tensor:
$$
J=  \left[ \begin {array}{cccc} -1&0&{\it a}&0\\ \noalign{\medskip}0&
1&0&{\it b}\\ \noalign{\medskip}0&0&1&0\\ \noalign{\medskip}0&0&0&
-1\end {array} \right],\quad
g=  \left[ \begin {array}{cccc} 0&0&1&0\\ \noalign{\medskip}0&0&0&-1
\\ \noalign{\medskip}1&0&-{\it a}&0\\ \noalign{\medskip}0&-1&0&-{
\it b}\end {array} \right].
$$

\subsection{Lie algebra $\rr_{4,-1,\beta}$} \label{G_8}
Nonzero Lie bracket: $[e_4,e_1] = e_1$, $[e_4,e_2] = -e_2$, $[e_4,e_3] = \beta e_3$.
Symplectic structure: $\omega =e^1\wedge e^2 +e^3\wedge e^4$, $-1< \beta <0$.
Three compatible para-complex structures $J$:

One para-K\"{a}hler structure with Hermitian Ricci tensor:
$$
J_1= \left[ \begin {array}{cccc} -1&0&0&0\\ \noalign{\medskip}0&1&0&0
\\ \noalign{\medskip}0&0&{\it a}&-{\frac {{{\it a}}^{2}-1}{{
\it b}}}\\ \noalign{\medskip}0&0&{\it b}&-{\it a}
\end {array} \right], \quad
RIC= \left[ \begin {array}{cccc} 0&0&0&0\\ \noalign{\medskip}0&0&0&0
\\ \noalign{\medskip}0&0&{\it b}\,{\beta}^{2}&0\\ \noalign{\medskip}0&0
&0&{\it b}\,{\beta}^{2}\end {array} \right].
$$

Two para-K\"{a}hler structure with zero Ricci tensor:
$$
J_2= \left[ \begin {array}{cccc} 1&0&0&0\\ \noalign{\medskip}{\it a}&-
1&0&0\\ \noalign{\medskip}0&0&1&{\it b}\\ \noalign{\medskip}0&0&0&
-1\end {array} \right], \quad
J_3= \left[ \begin {array}{cccc} -1&{\it a}&0&0\\ \noalign{\medskip}0&
1&0&0\\ \noalign{\medskip}0&0&1&{\it b}\\ \noalign{\medskip}0&0&0&
-1\end {array} \right].
$$

\subsection{Lie algebra $\rr_{4,-1,-1}$} \label{G_9}
Nonzero Lie bracket: $[e_4,e_1] = e_1$, $[e_4,e_2]= -e_2$, $[e_4,e_3] = -e_3$.
Symplectic structure: $\omega =e^1\wedge e^2 +e^3\wedge e^4$.
Five compatible para-complex structures $J$:

The para-K\"{a}hler structure with Hermitian Ricci tensor:
$$
J_1= \left[ \begin {array}{cccc} -1&-{\frac {{{\it a}}^{2}}{{\it b}}}&{\it a}&-{\frac {{\it a}\, \left( {\it c}-1 \right) }{
{\it b}}}\\ \noalign{\medskip}0&1&0&0\\ \noalign{\medskip}0&-{
\frac {{\it a}\, \left( {\it c}-1 \right) }{{\it b}}}&{
\it c}&-{\frac {{{\it c}}^{2}-1}{{\it b}}}
\\ \noalign{\medskip}0&-{\it a}&{\it b}&-{\it c}
\end {array} \right], \quad
RIC_1=  \left[ \begin {array}{cccc} 0&0&0&0\\ \noalign{\medskip}0&0&-{\it
a}&0\\ \noalign{\medskip}0&0&{\it b}&0\\ \noalign{\medskip}{
\it a}&0&0&{\it b}\end {array} \right]
$$

Three para-K\"{a}hler structures with zero Ricci tensor:
$$
J_2= \left[ \begin {array}{cccc} {\it a}&0&0&-{\frac {{{\it a}}^{2
}-1}{{\it b}}}\\
\noalign{\medskip}{\it c}&-{\it a}&{\it
b}&-{\it d}\\ \noalign{\medskip}{\it d}&-{\frac {{{\it
a}}^{2}-1}{{\it b}}}&{\it a}&{\frac {{\it c}\,{{\it
a}}^{2}-2\,{\it b}\,{\it d}\,{\it a}-{\it c}}{{{
\it b}}^{2}}}\\ \noalign{\medskip}{\it b}&0&0&-{\it a}
\end {array} \right], \quad
J_3=  \left[ \begin {array}{cccc} -1&{\it a}&0&0\\ \noalign{\medskip}0&
1&0&0\\ \noalign{\medskip}0&0&1&{\it b}\\ \noalign{\medskip}0&0&0&
-1\end {array} \right],\quad
J_4=  \left[ \begin {array}{cccc} -1&0&0&{\frac {2a}{{\it b
}}}\\ \noalign{\medskip}{\it b}&1&0&-{\it a}
\\ \noalign{\medskip}{\it a}&{\frac {2a}{{\it b}}}
&-1&{\it c}\\ \noalign{\medskip}0&0&0&1\end {array} \right].
$$

The para-K\"{a}hler structure of zero curvature:
$$
J_5= \left[ \begin {array}{cccc} 1&0&0&{\it a}\\ \noalign{\medskip}0&-1&0&0\\ \noalign{\medskip}0&{\it a}&1&{\it b}
\\ \noalign{\medskip}0&0&0&-1\end {array} \right],
g_5= \left[ \begin {array}{cccc} 0&-1&0&0\\ \noalign{\medskip}-1&0&0&-{
\it a}\\ \noalign{\medskip}0&0&0&-1\\ \noalign{\medskip}0&-{\it
a}&-1&-{\it b}\end {array} \right].
$$

\subsection{Lie algebra $\rr_{4,-\alpha,\alpha}$} \label{G_10}
Nonzero Lie bracket: $[e_4,e_1] = e_1$, $[e_4,e_2]= -\alpha e_2$, $[e_4,e_3] = \alpha e_3$.
Symplectic structure: $\omega =e^1\wedge e^4 +e^2\wedge e^3$, $0<\alpha <1$.
Three compatible para-complex structures $J$.

The para-K\"{a}hler structure with Hermitian Ricci tensor:
$$
J_1=   \left[ \begin {array}{cccc} -{\it a}&0&0&-{\frac {{{\it a}}^{
2}-1}{{\it b}}}\\ \noalign{\medskip}0&-1&0&0\\ \noalign{\medskip}0
&0&1&0\\ \noalign{\medskip}{\it b}&0&0&{\it a}\end {array}
 \right], \quad
RIC_1= \left[ \begin {array}{cccc} {\it b}&0&0&0\\ \noalign{\medskip}0&0
&0&0\\ \noalign{\medskip}0&0&0&0\\ \noalign{\medskip}0&0&0&{\it b}
\end {array} \right].
$$

Two para-K\"{a}hler structure with zero Ricci tensor:
$$
J_2= \left[ \begin {array}{cccc} 1&0&0&{\it a}\\ \noalign{\medskip}0&-
1&{\it b}&0\\ \noalign{\medskip}0&0&1&0\\ \noalign{\medskip}0&0&0&
-1\end {array} \right],\quad
J_3=\left[ \begin {array}{cccc} 1&0&0&{\it a}\\ \noalign{\medskip}0&1&0&0\\ \noalign{\medskip}0&{\it b}&-1&0\\ \noalign{\medskip}0&0&0&
-1\end {array} \right].
$$

\subsection{Lie algebra $\dd_{4,1}$} \label{G_11}
Nonzero Lie bracket: $[e_1,e_2] = e_3$, $[e_4,e_3] = e_3$, $[e_4,e_1] = e_1$.
Symplectic structures: $\omega_1 =e^1\wedge e^2 -e^3\wedge e^4$ and $\omega_2 =e^1\wedge e^2 -e^3\wedge e^4 +e^2\wedge e^4$.
\subsubsection{The case symplectic structure $\omega_1 =e^1\wedge e^2 -e^3\wedge e^4$}
Five para-K\"{a}hler structures:

The para-K\"{a}hler structure with Hermitian Ricci tensor
$$
J_{1,1}=  \left[ \begin {array}{cccc} -{\it a}&0&0&{\frac {{{\it a}}^{2
}-1}{{\it b}}}\\ \noalign{\medskip}{\it c}&{\it a}&{\it
b}&{\it d}\\ \noalign{\medskip}{\it d}&-{\frac {{{\it
a}}^{2}-1}{{\it b}}}&-{\it a}&-{\frac {{\it c}\,{{\it
a}}^{2}+2\,{\it a}\,{\it b}\,{\it d}-{\it c}}{{{
\it b}}^{2}}}\\ \noalign{\medskip}-{\it b}&0&0&{\it a}
\end {array} \right],\quad
RIC_{1,1}= \left[ \begin {array}{cccc} 0&0&2\,{\it b}&0\\ \noalign{\medskip}0
&0&0&0\\ \noalign{\medskip}0&0&0&0\\ \noalign{\medskip}0&-2\,{\it
b}&0&0\end {array} \right].
$$

The para-K\"{a}hler structure with zero Ricci tensor
$$
J_{1,2}= \left[ \begin {array}{cccc} 1&0&0&2\,{\frac {{\it a}}{{\it b}}}\\
\noalign{\medskip}{\it b}&-1&0&{\it a}\\
\noalign{\medskip}{\it a}&-2\,{\frac {{\it a}}{{\it b}}}&1&{\it c}\\ \noalign{\medskip}0&0&0&-1\end {array} \right], \quad
g_{1,2}= \left[ \begin {array}{cccc} {\it b}&-1&0&{\it a}
\\ \noalign{\medskip}-1&0&0&-2\,{\frac {{\it a}}{{\it b}}}
\\ \noalign{\medskip}0&0&0&1\\ \noalign{\medskip}{\it a}&-2\,{
\frac {{\it a}}{{\it b}}}&1&{\it c}\end {array} \right].
$$

The para-K\"{a}hler Einstein structure
$$
J_{1,3} =\left[ \begin {array}{cccc} 1&{\frac {{{\it a}}^{2}}{{\it b}}}&{\it a}&-{\frac {{\it a}\, \left( {\it c}+1 \right) }{{
\it b}}}\\ \noalign{\medskip}0&-1&0&0\\ \noalign{\medskip}0&{
\frac {{\it a}\, \left( {\it c}+1 \right) }{{\it b}}}&{
\it c}&-{\frac {{{\it c}}^{2}-1}{{\it b}}}
\\ \noalign{\medskip}0&{\it a}&{\it b}&-{\it c}
\end {array} \right], \quad
RIC_{1,3}= -\frac{3b}{2}\, Id.
$$

Two para-K\"{a}hler structures of zero curvature
$$
J_{1,4}= \left[ \begin {array}{cccc} -1&{\it a}&0&0\\ \noalign{\medskip}0&
1&0&0\\ \noalign{\medskip}0&0&1&{\it b}\\ \noalign{\medskip}0&0&0&
-1\end {array} \right],\quad
J_{1,5}= \left[ \begin {array}{cccc} -1&0&0&-{\it a}\\ \noalign{\medskip}0
&1&0&0\\ \noalign{\medskip}0&{\it a}&-1&{\it b}
\\ \noalign{\medskip}0&0&0&1\end {array} \right].
$$

\subsubsection{The case symplectic structure $\omega_2 =e^1\wedge e^2 -e^3\wedge e^4 +e^2\wedge e^4$.}
One para-K\"{a}hler structure of zero curvature
$$
J_{2,1}= \left[ \begin {array}{cccc} 1&{\it a}&0&0\\ \noalign{\medskip}0&-
1&0&0\\ \noalign{\medskip}0&0&-1&{\it b}\\ \noalign{\medskip}0&0&0
&1\end {array} \right], \quad
g_{2,1}= \left[ \begin {array}{cccc} 0&-1&0&0\\ \noalign{\medskip}-1&-{\it
a}&0&1\\ \noalign{\medskip}0&0&0&-1\\ \noalign{\medskip}0&1&-1&{
\it b}\end {array} \right].
$$

\subsection{Lie algebra $\dd_{4,2}$} \label{G_12}
Nonzero Lie bracket: $[e_1,e_2] = e_3$, $[e_4,e_3] = e_3$, $[e_4,e_1] = 2e_1$, $[e_4,e_2] = -e_2$.
Three symplectic structures: $\omega_1 =e^1\wedge e^2 -e^3\wedge e^4$, $\omega_2 =e^1\wedge e^4 +e^2\wedge e^3$ и $\omega_3 =e^1\wedge e^4 -e^2\wedge e^3$.
\subsubsection{The case symplectic structure $\omega_1 =e^1\wedge e^2 -e^3\wedge e^4$}
The Eistein para-K\"{a}hler structures:
$$
J_{1,1}= \left[ \begin {array}{cccc} -1&0&0&0\\
\noalign{\medskip}0&1&0&0\\ \noalign{\medskip}0&0&{\it a}&{\it b}\\ \noalign{\medskip}0
&0&-{\frac {{{\it a}}^{2}-1}{{\it b}}}&-{\it a}
\end {array} \right], \quad
g_{1,1}=\left[ \begin {array}{cccc} 0&1&0&0\\ \noalign{\medskip}1&0&0&0\\ \noalign{\medskip}0&0&{\frac {{{\it a}}^{2}-1}{{\it b}}}&{
\it a}\\ \noalign{\medskip}0&0&{\it a}&{\it b}\end {array}
 \right], \quad Ric_{1,1}=\frac {3(a^2-1)}{2\,b}\,g_{1,1}.
$$
Two para-K\"{a}hler Ricci-flat structures:
$$
J_{1,2}=\left[ \begin {array}{cccc} -1&{\it a}&0&0\\ \noalign{\medskip}0&
1&0&0\\ \noalign{\medskip}0&0&1&{\it b}\\ \noalign{\medskip}0&0&0&
-1\end {array} \right],\quad
g_{1,2}=\left[ \begin {array}{cccc} 0&1&0&0\\ \noalign{\medskip}1&-{\it a
}&0&0\\ \noalign{\medskip}0&0&0&1\\ \noalign{\medskip}0&0&1&{\it b
}\end {array} \right],
$$

$$
J_{1,3} =\left[ \begin {array}{cccc} 1&0&0&0\\ \noalign{\medskip}{\it a}&-
1&0&0\\ \noalign{\medskip}0&0&1&{\it b}\\ \noalign{\medskip}0&0&0&
-1\end {array} \right],\quad
g_{1,3} =\left[ \begin {array}{cccc} {\it a}&-1&0&0\\ \noalign{\medskip}-1
&0&0&0\\ \noalign{\medskip}0&0&0&1\\ \noalign{\medskip}0&0&1&{\it
b}\end {array} \right].
$$

\subsubsection{The case symplectic structure $\omega_2 =e^1\wedge e^4 +e^2\wedge e^3$}
Two para-K\"{a}hler structures:
$$
J_{2,1}= \left[ \begin {array}{cccc} {\it a}&0&0&2\,{\frac {{{\it a}}^
{2}-1}{{\it b}}}\\ \noalign{\medskip}0&-{\it a}&{\it b}&0
\\ \noalign{\medskip}0&-{\frac {{{\it a}}^{2}-1}{{\it b}}}&{
\it a}&0\\ \noalign{\medskip}-1/2\,{\it b}&0&0&-{\it a}
\end {array} \right],\quad
RIC_{2,1}=  \left[ \begin {array}{cccc} -3\,{\it b}&0&0&0
\\ \noalign{\medskip}0&0&0&0\\ \noalign{\medskip}0&0&0&0
\\ \noalign{\medskip}0&0&0&-3\,{\it b}\end {array} \right].
$$

$$
J_{2,2}=  \left[ \begin {array}{cccc} -{\it a}&0&0&-{\it b}\,(a+1)\\
\noalign{\medskip}0&1&0&0\\
\noalign{\medskip}0&{\it b}&-1&0\\
\noalign{\medskip}{\frac {{\it a}-1}{{\it b}}}&0
&0&{\it a}\end {array} \right],\quad
RIC_{2,2}= \left[ \begin {array}{cccc} 4\,{\frac {{\it a}-1}{{\it b}}}&0
&0&0\\ \noalign{\medskip}0&0&0&0\\ \noalign{\medskip}0&0&0&0
\\ \noalign{\medskip}0&0&0&4\,{\frac {{\it a}-1}{{\it b}}}
\end {array} \right].
$$

\subsubsection{The case symplectic structure $\omega_3 =e^1\wedge e^4 -e^2\wedge e^3$}
Nine para-K\"{a}hler structures of nonzero curvature and with Hermitian Ricci tensor:
$$
J_{3,1}= \left[ \begin {array}{cccc} -{\it a}&0&0&2\,{\frac {{{\it a}}^{2}-1}{{\it b}}}\\ \noalign{\medskip}0&{\it a}&{\it b}&0
\\ \noalign{\medskip}0&-{\frac {{{\it a}}^{2}-1}{{\it b}}}&-{
\it a}&0\\ \noalign{\medskip}-1/2\,{\it b}&0&0&{\it a}
\end {array} \right], \quad
RIC_{3,1}= \left[ \begin {array}{cccc} -3\,{\it b}&0&0&0\\ \noalign{\medskip}0&0&0&0\\ \noalign{\medskip}0&0&0&0
\\ \noalign{\medskip}0&0&0&-3\,{\it b}\end {array} \right].
$$

$$
J_{3,2}= \left[ \begin {array}{cccc} {\it a}&0&0&-{\it a}\,{\it b}
-{\it b}\\ \noalign{\medskip}0&-1&0&0\\ \noalign{\medskip}0&{\it
b}&1&0\\ \noalign{\medskip}{\frac {{\it a}-1}{{\it b}}}&0&0
&-{\it a}\end {array} \right],\quad
RIC_{3,2}= \left[ \begin {array}{cccc} 4\,{\frac {{\it a}-1}{{\it b}}}&0
&0&0\\ \noalign{\medskip}0&0&0&0\\ \noalign{\medskip}0&0&0&0
\\ \noalign{\medskip}0&0&0&4\,{\frac {{\it a}-1}{{\it b}}}
\end {array} \right].
$$

$$
J_{3,3}= \left[ \begin {array}{cccc} 1&{\it a}&0&{\it a}\\
\noalign{\medskip}0&1&0&0\\
\noalign{\medskip}2&{\it b}&-1&{\it a}\\ \noalign{\medskip}0&-2&0&-1\end {array} \right],\quad
RIC_{3,3}=  \left[ \begin {array}{cccc} 0&0&0&0\\ \noalign{\medskip}1&0&0&0
\\ \noalign{\medskip}0&0&0&0\\ \noalign{\medskip}0&0&-1&0\end {array}
 \right].
$$

$$
J_{3,4}= \left[ \begin {array}{cccc} 1&{\it a}&0&-{\it a}\\ \noalign{\medskip}0&1&0&0\\ \noalign{\medskip}-2&{\it b}&-1&{
\it a}\\ \noalign{\medskip}0&2&0&-1\end {array} \right],\quad
RIC_{3,4}= \left[ \begin {array}{cccc} 0&0&0&0\\ \noalign{\medskip}-1&0&0&0\\ \noalign{\medskip}0&0&0&0\\ \noalign{\medskip}0&0&1&0\end {array}
 \right].
$$

$$
J_{3,5}= \left[ \begin {array}{cccc} 0&{\it a}&1&{\it b}\\ \noalign{\medskip}0&0&0&-1\\ \noalign{\medskip}1&{\it b}&0&{
\it a}\\ \noalign{\medskip}0&-1&0&0\end {array} \right],\quad
g_{3,5}=\left[ \begin {array}{cccc} 0&-1&0&0\\ \noalign{\medskip}-1&-{\it
b}&0&-{\it a}\\ \noalign{\medskip}0&0&0&-1
\\ \noalign{\medskip}0&-{\it a}&-1&-{\it b}\end {array}
 \right],\quad Ric_{3,5}=0.
$$

$$
J_{3,6}=  \left[ \begin {array}{cccc} -1&0&0&0\\
\noalign{\medskip}{\it a}&1&{\it a}&0\\
\noalign{\medskip}0&0&-1&0\\
\noalign{\medskip}{\it a}&0&{\it a}&1\end {array} \right],\quad
RIC_{3,6}=  \left[ \begin {array}{cccc} 3\,{\it a}&0&-3/2\,{\it a}&0
\\ \noalign{\medskip}0&0&0&-3/2\,{\it a}\\ \noalign{\medskip}3/2\,
{\it a}&0&0&0\\ \noalign{\medskip}0&3/2\,{\it a}&0&3\,{\it
a}\end {array} \right].
$$

$$
J_{3,7}=\left[ \begin {array}{cccc} -1&0&0&0\\
\noalign{\medskip}{\it a}&1&-{\it a}&0\\
\noalign{\medskip}0&0&-1&0\\
\noalign{\medskip}-{\it a}&0&{\it a}&1\end {array} \right],\quad
RIC_{3,7}=\left[ \begin {array}{cccc} -3\,{\it a}&0&-3/2\,{\it a}&0
\\ \noalign{\medskip}0&0&0&-3/2\,{\it a}\\ \noalign{\medskip}3/2\,
{\it a}&0&0&0\\ \noalign{\medskip}0&3/2\,{\it a}&0&-3\,{\it
a}\end {array} \right].
$$

$$
J_{3,8}= \left[ \begin {array}{cccc} {\it a}-1&{\it b}&{\it a}&-2
\,{\frac {{\it b}\, \left( {\it a}-2 \right) }{{\it a}}} \\
\noalign{\medskip}1/2\,{\frac {{{\it a}}^{2}}{{\it b}}}&1/2
\,{\it a}+1&1/2\,{\frac {{{\it a}}^{2}}{{\it b}}}&-{\it
a}\\ \noalign{\medskip}-1/2\,{\it a}&-1/2\,{\frac {{\it b}\,{\it a}+4\,{\it b}}{{\it a}}}&-1/2\,{\it a}-1&{\it b}\\ \noalign{\medskip}1/2\,{\frac {{{\it a}}^{2}}{{\it b}
}}&1/2\,{\it a}&1/2\,{\frac {{{\it a}}^{2}}{{\it b}}}&-{\it a}+1\end {array} \right],
$$
$$
RIC_{3,8}=\left[ \begin {array}{cccc} 3/2\,{\frac {{{\it a}}^{2}}{{\it b}}}&0&-3/4\,{\frac {{{\it a}}^{2}}{{\it b}}}&0 \\ \noalign{\medskip}-3/4\,{\it a}&0&0&-3/4\,{\frac {{{\it a}}^{2}}{{\it b}}}\\ \noalign{\medskip}3/4\,{\frac {{{\it a}}^{2}}{{\it b}}}&0&0&0\\ \noalign{\medskip}0&3/4\,{\frac {{{\it a}}^{2}}{{\it b}}}&3/4\,{\it a}&3/2\,{\frac {{{\it a}}^{2}}{{\it b}}}\end {array} \right].
$$

$$
J_{3,9}=\left[ \begin {array}{cccc} -{\it a}-1&{\it b}&{\it a}&2
\,{\frac {{\it b}\, \left( {\it a}+2 \right) }{{\it a}}}
\\ \noalign{\medskip}1/2\,{\frac {{{\it a}}^{2}}{{\it b}}}&-1/
2\,{\it a}+1&-1/2\,{\frac {{{\it a}}^{2}}{{\it b}}}&-{\it
a}\\ \noalign{\medskip}-1/2\,{\it a}&-1/2\,{\frac {-{\it b
}\,{\it a}+4\,{\it b}}{{\it a}}}&1/2\,{\it a}-1&{\it
b}\\ \noalign{\medskip}-1/2\,{\frac {{{\it a}}^{2}}{{\it b
}}}&1/2\,{\it a}&1/2\,{\frac {{{\it a}}^{2}}{{\it b}}}&{
\it a}+1\end {array} \right],
$$
$$
RIC_{3,9}=\left[ \begin {array}{cccc} -3/2\,{\frac {{{\it a}}^{2}}{{\it
b}}}&0&-3/4\,{\frac {{{\it a}}^{2}}{{\it b}}}&0
\\ \noalign{\medskip}-3/4\,{\it a}&0&0&-3/4\,{\frac {{{\it a}}
^{2}}{{\it b}}}\\ \noalign{\medskip}3/4\,{\frac {{{\it a}}^{2}
}{{\it b}}}&0&0&0\\ \noalign{\medskip}0&3/4\,{\frac {{{\it a}}
^{2}}{{\it b}}}&3/4\,{\it a}&-3/2\,{\frac {{{\it a}}^{2}}{
{\it b}}}\end {array} \right].
$$

\subsection{Lie algebra $\dd_{4,\lambda}$} \label{G_13}
Nonzero Lie bracket: $[e_1,e_2] = e_3$, $[e_4,e_3] = e_3$, $[e_4,e_1] = \lambda e_1$, $[e_4,e_2] = (1-\lambda)e_2$, $\lambda\ge \frac 12$, $\lambda \neq 1,2$.
Symplectic structure: $\omega =e^1\wedge e^2 -e^3\wedge e^4$. Three para-K\"{a}hler structures of nonzero curvature.

Two Ricci-flat para-K\"{a}hler metrics:
$$
J_1= \left[ \begin {array}{cccc} 1&{\it a}&0&0\\ \noalign{\medskip}0&-
1&0&0\\ \noalign{\medskip}0&0&-1&{\it b}\\ \noalign{\medskip}0&0&0
&1\end {array} \right],\quad
g_1=\left[ \begin {array}{cccc} 0&-1&0&0\\
\noalign{\medskip}-1&-{\it a}&0&0\\
\noalign{\medskip}0&0&0&-1\\
\noalign{\medskip}0&0&-1&{
\it b}\end {array} \right].
$$

$$
J_2 =\left[ \begin {array}{cccc} 1&0&0&0\\ \noalign{\medskip}{\it a}&-1&0&0\\ \noalign{\medskip}0&0&1&{\it b}\\ \noalign{\medskip}0&0&0&
-1\end {array} \right],\quad
g_2=  \left[ \begin {array}{cccc} {\it a}&-1&0&0\\
\noalign{\medskip}-1 &0&0&0\\
\noalign{\medskip}0&0&0&1\\
\noalign{\medskip}0&0&1&{\it b}\end {array} \right].
$$

The para-K\"{a}hler Einstein metric:
$$
J_3= \left[ \begin {array}{cccc} 1&0&0&0\\
\noalign{\medskip}0&-1&0&0 \\
\noalign{\medskip}0&0&{\it a}&-{\frac {{{\it a}}^{2}-1}{{\it b}}}\\ \noalign{\medskip}0&0&{\it b}&-{\it a}
\end {array} \right],\quad
g_3= \left[ \begin {array}{cccc} 0&-1&0&0\\ \noalign{\medskip}-1&0&0&0
\\ \noalign{\medskip}0&0&-{\it b}&{\it a}\\ \noalign{\medskip}0
&0&{\it a}&-{\frac {{{\it a}}^{2}-1}{{\it b}}}\end {array}
 \right],\quad Ric_3=-\frac {3b}{2}\, g_3.
$$

\subsection{Lie algebra $\h_{4}$} \label{G_14}
Nonzero Lie bracket: $[e_1,e_2] = e_3$, $[e_4,e_3] = e_3$, $[e_4,e_1] = \frac 12 e_1$, $[e_4,e_2] = e_1+ \frac 12 e_2$.
Symplectic structure: $\omega =\pm (e^1\wedge e^2 -e^3\wedge e^4)$.
One para-K\"{a}hler structure with zero Ricci tensor
$$
J=  \left[ \begin {array}{cccc} -1&{\it a}&0&0\\ \noalign{\medskip}0&
1&0&0\\ \noalign{\medskip}0&0&1&{\it b}\\ \noalign{\medskip}0&0&0&
-1\end {array} \right],\quad
g=  \left[ \begin {array}{cccc} 0&1&0&0\\ \noalign{\medskip}1&-{\it a
}&0&0\\ \noalign{\medskip}0&0&0&1\\ \noalign{\medskip}0&0&1&{\it b
}\end {array} \right].
$$

\subsection{Lie algebra $\R^{4}$} \label{G_15}
Symplectic structure: $\omega =e^1\wedge e^2 +e^3\wedge e^4$ на $\mathbb{R}^4$.  Any para-complex structure $J$ on $\mathbb{R}^4$, compatible with $\omega$ defines a para-K\"{a}hler structure of zero curvature.

\section{Paracontact structures} \label{Paracontact}
An {\em almost paracontact structure} on a $(2n+1)$-dimensional (connected) smooth manifold $M$ is given by a triple $(\varphi,\xi,\eta)$, where $\varphi$ is a $(1,1)$-tensor, $\xi$ a global vector field and $\eta$ a $1$-form, such that
 \begin{equation}\label{almostpc}
\varphi (\xi)=0, \qquad \eta \circ \varphi =0, \qquad \eta (\xi)=1,  \qquad \varphi ^2 = Id -\eta \otimes \xi
\end{equation}
and the restriction $J$ of $\varphi$ on the horizontal distribution $D={\rm ker}(\eta)$ is an almost paracomplex structure (that is, the eigensubbundles $D^+, D^-$ corresponding to the eigenvalues $1,-1$ of $J$ have equal dimension $n$).

If $\eta$ satisfies $\eta\wedge (d\eta)^n \neq 0$, then $D$ defines a contact structure on $M$. In this case, $\eta$ is called a contact form, and $\xi$ is called the Reeb vector field. In this case, $\xi$ satisfies $d\eta(\xi,X)=0$ for any vector field $X$ on $M$.

A pseudo-Riemannian metric $g$ on $M$ is {\em compatible} with the almost paracontact structure $(\varphi,\xi,\eta)$ if
\begin{equation}\label{compg}
g(\varphi X, \varphi Y) =-g(X,Y) +\eta(X)\eta(Y).
\end{equation}

In such a case, $(\varphi,\xi,\eta,g)$ is said to be an {\em almost paracontact metric structure}.
Observe that $\eta(X)=g(\xi,X)$ for any compatible metric. Any almost paracontact structure admits compatible metrics, which have signature $(n+1,n)$. The {\em fundamental $2$-form} $\Phi$ of an almost paracontact metric structure $(\varphi,\xi,\eta,g)$ is defined by $\Phi(X,Y)=g(\varphi X,Y)$, for all tangent vector fields $X,Y$. If $\Phi=d\eta$,
then the manifold $(M,\eta,g)$ (or $(M,\varphi,\xi,\eta,g)$) is called a {\em paracontact metric manifold} and $g$ the {\em associated metric}.

It is easy to see that the associated metric $g$ is completely determined by the affinor $\varphi$:
\begin{equation}\label{compg}
g(X,Y) = d\eta(\varphi X,Y) + \eta(X)\eta(Y).
\end{equation}

Consider the product manifold $M\times \mathbb{R}$ and let $(X, f \partial_t)$ denote an arbitrary vector field on $M\times \mathbb{R}$. Then,
\begin{equation}\label{Norm}
\mathcal{J}(X, f \partial_t)=(\varphi X -f\xi,-\eta(X)\partial_t )
\end{equation}
defines an almost para-complex structure on $M\times \mathbb{R}$. When $\mathcal{J}$ is integrable, $(\varphi,\eta,\xi)$ is said to be normal.

The normal paracontact metric structure $(\eta,\xi, \varphi, g)$ is called para-Sasakian.

\subsection{Paracontact structures on central extensions of Lie algebras} \label{ParacontactLie}
A left-invariant contact structures on Lie groups $H$ are determined by their values on the Lie algebra $\h$.
In this sense, we will talk about contact Lie algebras $(\h, \eta)$.
As is known, \cite{Diatta} a contact Lie algebra with a nontrivial center is the central extension of symplectic Lie algebra $(\g, \omega)$ using the 2-cocycle $\omega$.
Recall the central extension procedure.

If we have a symplectic Lie algebra $(\g, \omega)$, then the central extension is a Lie algebra $\h=\g\times_\omega \mathbb{R}$ in which the Lie brackets are defined as follows:
\begin{enumerate}
  \item $[X, \xi]_{\h} = 0$,
  \item $[X, Y]_{\h} = [X, Y]_{\g} +\omega(X, Y)\xi$ for any $X, Y \in \g$,
\end{enumerate}
where $\xi=\partial_t$.
On the Lie algebra $\h=\g\times_\omega \mathbb{R}$ a contact form is given by the form $\eta =\xi^*$, and $\xi=\partial_t$ is the Reeb vector field.
The contact distribution $D$ corresponds to the subspace $\g \subset \h$.  If $x = X +\lambda\xi$ and $y = Y +\mu\xi$, where  $Y,Y \in \g$, $\lambda, \mu \in \mathbb{R}$,  then:
$$
d\eta (x,y) = -\eta([x, y]) = -\xi^*([X, Y]_{\g} +\omega(X, Y)\xi) =-\omega(X, Y).
$$
The affinor $\varphi$ of a paracontact structure is defined by its action on the contact subbundle $D =\g$.
Therefore, to define the affinor $\varphi$ on $\h=\g\times_\omega \mathbb{R}$, we can use almost para-complex structure $J$ on $\g$ as follows: if $x = X +\lambda\xi$, where $X\in\g$, then $\varphi(x) =JX$. Thus, the affinor $\varphi$ has a block form:
$$
\varphi=\left(
          \begin{array}{cc}
            J & 0 \\
            0 & 0 \\
          \end{array}
        \right).
$$
If the almost para-complex structure $J$ on $\g$ is also compatible with $\omega$, i.e., it has the property $\omega(JX,JY) = -\omega(X,Y)$, then we obtain a paracontact (pseudo-Riemannian) metric structure $(\eta, \xi, \varphi, h)$  on $\h=\g\times_\omega \mathbb{R}$, where $h$ is associated metric:
$$
h(x,y) =-d\eta(x, \varphi y) + \eta(x)\eta(y).
$$

As is known \cite{Smolen-19}, the paracontact metric structure $(\eta, \xi, \varphi, h)$ on the central extension $\h=\g\times_\omega \mathbb{R}$ is para-Sasakian if and only if the symplectic algebra $(\g,\omega,J,g)$ is para-K\"{a}hler.
As a consequence, we obtain the following conclusion: {\it the above classification of para-K\"{a}hler structures on four-dimensional Lie algebras gives a classification of para-Sasakian structures on five-dimensional Lie algebras with a nontrivial center.}

Note that in \cite{Calvar-Perron}, another version of the classification based on \cite{Calvar-15} was proposed.

It was shown in \cite{Smolen-19} that the curvature tensor of the para-Sasakian structure on the central extension $\h=\g\times_\omega \mathbb{R}$ is expressed through the curvature tensor of the para-K\"{a}hler structure on $\g$.

\begin{theorem} [\cite{Smolen-19}] \label{T2}
Suppose $(\omega, J,g)$ is the para-K\"{a}hler structure on the Lie algebra $\g$ and $(\eta, \xi, \varphi, h)$ is the corresponding contact para-Sasakian structure on the central extension $\h=\g\times_\omega \mathbb{R}$. Then the curvature tensor $R$ on $\h$ is expressed in terms of the curvature tensor $R_{\g}$ on $\g$, the form $\omega$ and the almost para-complex structure $J$ on $\g$ as follows: for any $X,Y\in\g$,
$$
R(X,Y)Z = R_{\h}(X,Y)Z -\frac 14 g(X,JZ)JY + \frac 14 g(Y,JZ)JX -\frac 12 g(X,JY)JZ,
$$
$$
R(X,Y)\xi = 0, \quad R(X, \xi) Z = \frac 14 g(X,Z)\xi, \quad R(X,\xi)\xi = -\frac 14 X,
$$
\end{theorem}

\begin{theorem} [\cite{Smolen-19}] \label{T3}
Suppose $(\omega, J,g)$ is the para-K\"{a}hler structure on the Lie algebra $\g$ and $(\eta, \xi, \varphi, h)$ is the corresponding contact para-Sasakian structure on the central extension $\h=\g\times_\omega \mathbb{R}$.
Then the Ricci tensor $Ric$ on $\h$ is expressed in terms of the Ricci tensor $Ric_{\g}$ on $\g$, the form $\omega$ and the almost para-complex structure $J$ on $\g$ as follows:
$$
Ric(Y,Z) = Ric_{\g}(Y,Z)+\frac 12 g(Y,Z),\quad Ric(Y,\xi)=0, \quad Ric(\xi,\xi)=-\frac n2.
$$
\end{theorem}

The classification expressions of para-complex structures $J$ on $\h$ obtained in Section \ref{Lie_algebras} (Theorem \ref{T1}) and the formulas of the curvature tensor in Theorems \ref{T2} and \ref{T3} make it easy to obtain the curvature characteristics of para-Sasakian structures on central extensions $\h=\g\times_\omega \mathbb{R}$ in the five-dimensional case.

\end{document}